\documentclass[12pt, reqno]{amsart}

\usepackage{amsfonts,amsmath}
\usepackage{amscd}
\usepackage{amssymb}

\allowdisplaybreaks

\usepackage{hyperref}

\date{\today}

\newtheorem{thm}{Theorem}[section]

\newtheorem{lem}[thm]{Lemma}

\newtheorem{prop}[thm]{Proposition}
\theoremstyle{definition}

\theoremstyle{remark}
\newtheorem{rem}[thm]{Remark}

\numberwithin{equation}{section}

\newcommand{\R}{\mathbb R}

\newcommand{\Z}{\mathbb Z}

\newcommand{\Na}{\mathbb N}
\newcommand{\He}{\mathbb H}

\newcommand{\C}{{\mathbb C}}

\renewcommand{\Im}{\operatorname{Im}}

\usepackage{color}

\title[Chernoff and Ingham type theorems]
{ Theorems of Chernoff and Ingham \\ for certain eigenfunction expansions}

\author[ Ganguly and Thangavelu]{ Pritam Ganguly and Sundaram Thangavelu}

\address[P. Ganguly, S. Thangavelu]{Department of Mathematics\\
Indian Institute of Science\\
560 012 Bangalore, India}

\email{pritamg@iisc.ac.in, veluma@iisc.ac.in}

\date{}
 \keywords{Chernoff's theorem, Compact Symmetric Spaces, Hermite and special Hermite expansions, Ingham's theorem, Uncertainty principle}
 \subjclass[2010]{Primary:  43A85, 42C05. Secondary: 33C45, 35P10}
 
\begin{document}

\maketitle

\begin{abstract}  We prove an uncertainty principle for certain eigenfunction expansions on $ L^2(\mathbb{R}^+,w(r)dr) $ and use it to prove analogues of theorems of  Chernoff and Ingham for 
Laplace-Beltrami operators on compact symmetric spaces, special Hermite operator on $ \mathbb{C}^n $ and Hermite operator on $ \mathbb{R}^n.$
 \end{abstract}

\section{Introduction} Suppose $ f $ is a non trivial function from $ L^1(\R^n) $ which vanishes on an open set $ V.$  Then what is the best possible admissible decay  for the Fourier transform $ \hat{f}(\xi)$  defined by
$$ \hat{f}(\xi) = (2\pi)^{-n/2} \int_{\R^n}  e^{-i x \cdot \xi}\, f(x)\, dx \,\, ?$$ 
Under the added assumption that $ \hat{f} \in L^1(\R^n) $ we have the inversion formula 
$$  f(x)  = (2\pi)^{-n/2} \int_{\R^n}  e^{i x \cdot \xi}\, \hat{f}(\xi) \, dx$$ 
and from the easy part of the well known Paley-Wiener theorem, it is clear that $ \hat{f} $ cannot have compact support. A moment's thought staring at the inversion formula also reveals that $ \hat{f}$ cannot have any exponential decay. For, if $ |\hat{f}(\xi)| \leq C e^{-a|\xi|} $ then $ f $ extends as a holomorphic function in an open neighbourhood of $ \R^n $ in $ \C^n $ and hence cannot vanish on any open set. Therefore, it is natural to ask what is the best possible estimate of the form $ |\hat{f}(\xi)| \leq C e^{-\psi(|\xi|)} $ that is compatible with vanishing of $ f $ on an open set.\\

In the one dimensional case, this problem has been addressed by Ingham \cite{I}, Levinson \cite{Levinson}, Paley and Wiener \cite{PW1,PW2}  and their results are in terms of  non integrability of $ \psi(t) \, t^{-2} $ over $ [1,\infty ).$ Recently this problem has received considerable attention and several versions have been proved  in the contexts of $ \R^n,$ nilpotent Lie groups and  compact and non-compact Riemannian symmetric spaces, see the works \cite{BSR}, \cite{BPR}, \cite{BS},\cite{BR} of Bhowmik, Pusti, Ray and  Sen in various combinations of authorship. Recently, in a joint work with Bagchi and Sarkar \cite{BGST} we have proved an analogue of Ingham's theorem for the operator valued Fourier transform on the Heisenberg group.\\  

For the convenience of the reader, let us recall the original version of Ingham's theorem for the Fourier transform on $ \R $ proved in \cite{I}.
	\begin{thm}[Ingham]\label{ingh}
		Let $ \theta(y) $ be a nonnegative even function on $\R$ such that $ \theta(y) $ decreases to zero when $ y \rightarrow \infty.$ There exists a nonzero continuous function $ f $ on $ \R,$ equal to zero outside an interval $ (-a,a) $ having Fourier transform $ \widehat{f} $ satisfying the estimate $ |\widehat{f}(y)|\leq C e^{-|y|\,\theta(y)} $ if and only if $ \theta $ satisfies $ \int_1^\infty \theta(t) t^{-1} dt <\infty.$
	\end{thm} 

The original proof of Ingham for the Fourier transform on $ \R $  reduced matters to the celebrated theorem of Denjoy and Carleman on quasi analytic  functions. The $ n$-dimensional version in the case of $ \R^n $ can be  proved using Chernoff's theorem \cite{CH2}  which is a higher dimensional analogue  of the Denjoy-Carleman theorem. The proof presented in \cite{BPR} in the context of non-compact symmetric spaces depends on a higher dimensional version of Carleman's theorem due to de Jeu \cite{J} which is closely related to determinacy of measures and Carleman condition on its  moments. Even though all the results mentioned above are formulated in terms of the Fourier transform on the underlying Lie group, the proofs  are invariably based on spectral properties of the associated Laplace-Beltrami operators.\\

Let us now consider the general setting of a  strictly positive elliptic partial differential operator $ P $ of order 2 on a Riemannian manifold $ \Omega.$ Assume that the spectrum of $ P $ is discrete consisting of distinct points  $  \lambda_{k+1} > \lambda_k\geq 0.$  Denoting the spectral projections associated to $ \lambda_k $ by $ P_k $ we have the eigenfunction expansion
$$  f = \sum_{k=0}^\infty  P_kf , \,\,\, f \in L^2(\Omega).$$ 
Assuming that $ f $ is nontrivial and vanishing on an open set $ V ,$ we can ask for the best possible decay of  $ \|P_k f\|_2 $ as a function of the eigenvalue $ \lambda_k.$  Obviously, any decay that leads us to conclude that $ f $ is real analytic, is not admissible. For example, in view of a theorem of Kotake and Narasimhan \cite{KN}, any decay that leads to the estimates 
$$  \| P^m f\|_2 = \big(  \sum_{k=0}^\infty \lambda_k^{2m} \|P_kf\|_2^2 \big)^{1/2} \leq (2m)! C^m $$ 
for all $ m \in \Na$ is not allowed.  As can be easily checked, under the assumption that $ \lambda_k $ grows like $ k^2 $ as $ k \rightarrow \infty,$ exponential decays like $ \|P_kf\|_2 \leq C e^{-a\sqrt{\lambda_k}} $ are ruled out.  As in the case of Lie groups, we look for functions $ \psi $ such that $ \|P_k f\|_2 \leq C e^{-\psi(\sqrt{\lambda_k})} $ are admissible.\\

In this paper we study the above problem in the following three settings: (i)  The Laplace-Beltrami operator $ \Delta $ on a compact Riemannian symmetric space $ X,$ (ii) the special Hermite operator $ L $ on $\C^n$ where the projections are infinite dimensional and (iii) the Hermite operator $ H $ on $ \R^n.$ In all these cases we prove analogues of Chernoff's theorem for the associated operators and then use the same to prove a version of Ingham's theorem.

\begin{thm}
	\label{chg}
	 Let $ (\Omega, P ) $ be any one of the three pairs listed above. Suppose $ f \in C^{\infty}(\Omega) $ is such that $P^mf\in L^2(\Omega)$ for all $m\geq 0$ and the sequence $ \| P^m f\|_2 $ satisfies the  Carleman condition $ \sum_{m=1}^\infty   \| P^m f\|_2^{-1/2m} = \infty.$  Then $ f $ cannot vanish on any open subset of $ \Omega $ unless  it is identically zero.
\end{thm}

By making use of spherical means, we reduce matters to proving an analogue of the above theorem  for eigenfunction expansions for the radial part  of $ P $ realised on the space $ L^2(A,w(r)dr) $ where $ A \subset (0,\infty) $ is an interval and $ w $ is a positive weight function. We also remark that the Carleman condition can be considerably weakened. We then obtain the  following analogue of Theorem \ref{ingh}: 
\begin{thm}\label{ingh_L2}
	 Let $ (\Omega, P ) $ be any one of the three pairs listed above. Let $ \theta(t) $ be a nonnegative even function on $\R$ such that $ \theta(t) $ decreases to zero when $ t \rightarrow \infty.$ There exists a nonzero compactly supported continuous function $ f $ on $ \Omega$  satisfying 
	 \begin{equation} 
	 \|P_kf\|_2 \leq C e^{-\sqrt{\lambda_k}\theta(\sqrt{\lambda_k})},
	  \end{equation}
for all $ k \in \Na$	 if and only if the function $ \theta $ satisfies $ \int_1^\infty \theta(t) t^{-1} dt <\infty.$
\end{thm}

 Under the assumption that  $ \int_1^\infty \theta(t) t^{-1} dt =\infty ,$ the above theorem rules out the decay  (1.1) for any compactly supported function $ f.$  Actually we can prove more: any function for which (1.1) is valid cannot vanish on any non empty open set.  We can also replace  the decay  condition  by a much weaker pointwise decay of $P_kf(g).$ Indeed, we have the following version of Ingham's theorem. 

\begin{thm} Let $ (\Omega, P ) $ be any one of the three pairs listed above. Let $ \theta $ be a positive decreasing function on $ [0,\infty) ,$ which vanishes at  infinity. Assume that $ \int_1^\infty \theta(t) \, t^{-1} dt = \infty.$  If $ f \in L^2(\Omega) $ vanishes on an open set $ V ,$ then it cannot satisfy the decay condition $ |P_kf(g)| \leq C_g e^{- \sqrt{\lambda_k}\, \theta(\sqrt{\lambda_k})} $ for all $ g \in V $ unless it is identically zero.
\end{thm}

Note that the pointwise decay condition on $ P_kf(g) $ is assumed only on $ V $ which is much weaker than the condition on $ \|P_kf\|_2 $ even when $ \Omega$ is compact. Interestingly, the proof only requires the one dimensional result of Carleman in the form stated by de Jeu in \cite{J}. In the case of Hermite expansions, we also give a different proof of the above theorem under a slightly different assumption on the behaviour of $ \theta .$\\

We remark that analogues of Theorem 1.4 can be established  for the Laplace-Beltrami operator on non-compact Riemannian symmetric spaces and the Dunkl Laplacian on $ \R^n,$   see \cite{GT}. \\

This paper is organised as follows. In Section 2 we prove Theorem \ref{chg} in a general setting. Using this we prove analogues of theorems of Chernoff and Ingham for Laplace-Beltrami operators on rank one compact symmetric spaces, special Hermite operator on $\C^n$ and Hermite operator on $\R^n$ respectively in Sections 3, 4 and 5.  %We also remark that in the case of compact symmetric spaces we work with  a modified Laplace-Beltrami operator whose eigenvalues are perfect squares. This  is  mainly done for cosmetic reasons and the same proof with minor modifications works for the unmodified operator. 
 Finally, in Section 6 we provide examples of compactly supported functions, the norms of whose spectral projections have Ingham type decay, thus  proving  the sharpness of Ingham's theorem and completing the proof of Theorem \ref{ingh_L2} in three different settings under consideration in this paper.

\section{Chernoff's theorem for  eigenfunction expansions}
 Given a positive weight function $ w(r), r \in \R^+ $ consider an orthogonal basis $ \psi_k , k \in \Na$ for  the Hilbert space $ L^2(\R^+, w(r) dr)$  where we assume that $ \psi_k(0) = 1.$ Assume that $ \psi_k $ are eigenfunctions of an elliptic differential operator $ P $ on $ \R^+ $ with eigenvalues  $ \lambda_k  \geq 0 $ which goes to infinity as $ k \rightarrow \infty.$ For $ f \in L^2(\R^+, w(r)dr) $ we have the norm-convergent expansion 
 \begin{equation}\label{expan}  f(r)  = \sum_{k=0}^\infty   c_k \, \hat{f}(k) \psi_k(r),\,\,\,  \hat{f}(k) =  \int_0^\infty f(r) \psi_k(r) w(r) dr \end{equation}
 where $ c_k^{-1}  = \int_0^\infty |\psi_k(r)|^2 w(r) dr.$ The Parseval's identity for the above expansions gives, for any $ m \in \Na $ the following: 
 $$ \| P^m f\|_2^2  = \sum_{k=0}^\infty  \lambda_k^{2m} \, c_k\, |\hat{f}(k)|^2 .$$ 
 Suppose we further assume  that both the sequences $ \lambda_k $  and $ c_k $ grow polynomially so that $ \sum_{k=1}^\infty \lambda_k^{-N} <\infty $ for large enough $ N.$ Then under the assumption that $ P^m f \in L^2(\R^+,w(r)dr) $ for a sufficiently large $ m $  we can prove that the series (\ref{expan}) converges uniformly and  hence $ f $ is continuous.\\
 
 Let $ f $ be such that $ P^m f \in L^2(\R^+,w(r)dr) $ for all $ m.$   Now we associate a Borel measure $ \mu_f $ on $ \R $ defined in terms of the sequence $( \hat{f}(k) ) $ as follows: for any  Borel function  $ \varphi $  on $ \R$ 
 $$ \int_{-\infty}^\infty \varphi(t) d\mu_f(t) =   \frac{1}{2} \sum_{k=0}^\infty c_k\, \big( \varphi(\sqrt{\lambda_k})+\varphi(-\sqrt{\lambda_k})\big) |\hat{f}(k)|.$$  Observe that  $ \int_{-\infty}^\infty \varphi(t) d\mu(t) =\int_{-\infty}^\infty \varphi(-t) d\mu(t) $ and hence $ \int_{-\infty}^\infty \varphi(t) d\mu_f(t) = 0 $ for all odd functions $ \varphi .$ Under the assumption that $ f $ vanishes on the interval $ 0 < r < \delta,$ the convergence of the series (\ref{expan}) to zero on the interval $ (0,\delta) $ allows us to apply $ P^m $ term by term to conclude that
 \begin{equation}\label{expan-1}
 P^mf(r)  = \sum_{k=0}^\infty \lambda_k^{m} \, c_k\, \hat{f}(k) \psi_k(r) = 0 
  \end{equation} 
 for $ 0<r<\delta.$  As $\psi_k(0)=1$ the equation  (\ref{expan-1}) implies that  
 \begin{equation}\label{exp-2}  \sum_{k=0}^\infty \lambda_k^{m} \, c_k\, \hat{f}(k)  = 0 ,\,\, m \in \Na .\end{equation}
  We now look for conditions on the function $ f $ so that polynomials are dense in  $ L^1(\R, d\mu_f ).$   \\
 
 The following result in conjunction with the above discussion gives rise to an analogue of Chernoff's theorem for the expansion (\ref{expan}).
 
 \begin{thm}\label{carleman} Let $ \mu $ be a finite positive Borel measure on $ \R $ for which all the moments $ M(m) =\int_{-\infty}^\infty t^m d\mu $ are finite. If we further assume that the moments satisfy the Carleman condition $ \sum_{m=1}^\infty  M(2m)^{-1/2m} = \infty,$ then polynomials are dense in $ L^p(\R,d\mu), 1 \leq p < \infty.$ 
 \end{thm}
 
 \begin{rem}  If we  assume that $ \mu $ is even in the sense that $ \int_{-\infty}^\infty \varphi(t) d\mu(t) =\int_{-\infty}^\infty \varphi(-t) d\mu(t) ,$ then even polynomials are dense in the subspace $ L^p_{even}(\R, d\mu).$
 \end{rem}
 The above theorem has an $n$-dimensional version, see de Jeu \cite{J}. However, for many applications  the  one dimensional version suffices.
 We also require the following elementary result about series of positive real numbers proved in \cite[Lemma 3.3 ]{BPR}
 \begin{lem}
 	\label{serieslem}
 	Let $\{a_m\}_m$ be a sequence of positive real numbers such that $\sum_{m=1}^{\infty} a_m=\infty$, then for any positive integer $j$, we have $\sum_{m=1}^{\infty} a_m^{1+\frac{j}{m}}=\infty.$
 \end{lem}
  Here is a version of Chernoff's  theorem for the operator $ P.$ We assume that the elliptic differential operator $ P $ is such that the associated eigenvalues  $ \lambda_k $ and eigenfunctions $ \psi_k $ satisfy the following conditions: (i) $ \psi_k(0) =1,$ (ii) $ \lambda_k \geq 0,$ and (iii) $ \lambda_k $ and $ c_k $  have polynomial growth in $ k.$ Closely following the work of  Bhowmik-Pusti-Ray \cite{BPR}, we establish the following result.
 
 \begin{thm}\label{chernoff} Let $ f \in L^2(\R^+, w(r)dr) $ be such that $ P^mf \in L^2(\R^+, w(r)dr) $ for all $ m \in \Na $  and satisfies  the Carleman condition 
 $ \sum_{m=1}^\infty  \| P^m f \|_2^{-1/(2m)} = \infty.$ Then $ f $ cannot vanish in a neighbourhood of $0 $ unless it is identically zero.
 \end{thm}
 \begin{proof} Let $f$ be as in the statement of the theorem and $\mu_f$ be the corresponding  measure defined as above. Let $M(m)$ stand for the $m$ th moment of $\mu_f$. Then for any $m\geq 1$ we have 
 	$$M(2m)=\int_{-\infty}^{\infty}t^{2m}d\mu_{f}(t)=\sum_{k=0}^{\infty}\lambda_k^mc_k |\hat{f}(k)|.$$ First assume that $ \lambda_0 > 0 .$ By writing 
 	 $$M(2m)=\sum_{k=0}^{\infty}\lambda_k^{-j}\lambda_k^{m+j}c_k|\hat{f}(k)|$$
 and  applying Cauchy-Schwarz inequality  we get the estimate
  $$ M(2m)  \leq  \sqrt{C_j}\left(\sum_{k=0}^\infty  \lambda_k^{2(m+j)} \, c_k \, |\hat{f}(k)|^2\right)^{\frac12} =    \sqrt{C_j} \,  \| P^{m+j} f \|_2 $$ 
 where $ C_j = \sum_{k=0}^\infty \lambda_k^{-2j} c_k < \infty $ if $ j $ is large enough under the assumption that $ c_k $ has polynomial growth.  When $ \lambda_0 = 0 $ the same estimate is true with $ C_j = \sum_{k=1}^\infty \lambda_k^{-2j} c_k .$  From the above we obtain
 $$ \sum_{m=1}^\infty  M(2m)^{-1/2m} \geq  \sum_{m=1}^\infty ( \sqrt{C_j})^{-1/2m} \| P^{m+j}f \|_2^{-1/2m} =\sum_{m=1}^\infty ( \sqrt{C_j})^{-1/2m} \left(\| P^{m+j}f \|_2^{-\frac{1}{2(m+j)}}\right)^{\frac{m+j}{m}}. $$ 
 Therefore, by Lemma \ref{serieslem}, the divergence of the series on the right hand follows from the divergence of $\sum_{m=1}^\infty  \| P^{m}f\|_2^{-1/2m}$  which is the hypothesis.  Therefore, the  moments of the measure $ \mu_f $ satisfy the condition in Theorem \ref{carleman}  and hence we conclude that polynomials are dense in $ L^1(\R,d\mu_f).$   Consider the even function $\varphi$ defined on the support of $\mu_f$  by $\varphi(\sqrt{\lambda_k}) = \varphi(-\sqrt{\lambda_k}) = \overline{\hat{f}(k)}.$ We observe that $$\int_{-\infty}^{\infty}|\varphi(t)|d\mu_{f}(t)= \sum_{k=0}^{\infty}c_k|\hat{f}(k)|^2<\infty$$ which proves that $\varphi \in L^1(\mathbb{R},d\mu_f).$  As $ \varphi $ is even, for  any $ \epsilon > 0 $ we can choose an even polynomial $ q $ such that  $\|\varphi-q\|_{L^1(\mathbb{R},d\mu_{f})}<\epsilon$  which translates into
 $$ \sum_{k=0}^\infty |\overline{\hat{f}(k)}-q(\sqrt{\lambda_k})| \, c_k \,  |\hat{f}(k)| < \epsilon.$$ Now if we assume that $f$ vanishes in a neighbourhood of zero, then 
 as $ q $ is even, (\ref{exp-2}) allows us to conclude that $  \sum_{k=0}^\infty  q(\sqrt{\lambda_k}) \, c_k\,  \hat{f}(k) \, = 0.$ 
 By writing 
 $$ |\hat{f}(k)|^2 = \big(  \overline{\hat{f}(k)}- q(\sqrt{\lambda_k}) \big) \hat{f}(k) + q(\sqrt{\lambda_k}) \hat{f}(k)$$ and  making use of  the above observations, we conclude that
 $  \|f\|_2^2  = \sum_{k=0}^\infty c_k\,  |\hat{f}(k)|^2  < \epsilon .$
As this is true for every  $ \epsilon $ it follows that $ f  = 0 $ proving the theorem.
 \end{proof}
\begin{rem}
	 It is easy to see from the above proof that  Theorem \ref{chernoff} still holds true  under the weaker assumption that $ \displaystyle\lim_{r \rightarrow 0} P^mf(r) = 0 $ for all $ m\geq 0.$
\end{rem}
 
 In the following sections we discuss several examples of operators $ P $  for which the above theorem applies. As applications we also prove analogues of  Chernoff and Ingham  theorems in various settings.

\section{Jacobi polynomials and Compact symmetric spaces}

Let $ (G,K) $ be a compact symmetric space and let $ X = G/K $ be the associated symmetric space.  We assume that $ X $ has rank one. Let  $ G = KAK $ be  a Cartan decomposition of $ G $ where $ A $ is identified with $ (0, R)$  for some $ R>0.$ The elements of $ A $ will be denoted by $ a_r, r>0.$  Let  $ dk $ be the normalised Haar measure on $ K.$ Then the  Haar measure on $ G $ has the  decomposition  $ dg = w(r) dk dr dk^\prime $ for some positive weight function $ w $ on $A.$ Given $ f \in L^1(G) $ we define the spherical means of $ f $ by
\begin{equation}\label{spher-mean}  f(g,r) = \int_K \int_K f(gka_rk^\prime) dk\, dk^\prime.\end{equation}
Observe that $ f(g,r) $ is a right $ K$-invariant function of $ g \in G $ and hence we can consider it as a function on the symmetric space $ X.$ The spherical means $ f(g,r) $ can be realised  as $ f \ast \nu_r $ where $ \nu_r $ is a compactly supported $K$-biinvariant probability measure on $ G.$ We refer to the paper of Pati-Shahshahani-Sitaram \cite{PSS} for more about the spherical means.\\

Using Peter-Weyl theorem we can expand $ f(g,r) $ in terms of zonal spherical functions. Recall that for right $K$-invariant functions $f $ on $ G $ the Peter-Weyl theorem reads as
\begin{equation}\label{peter-weyl}  f(g)  = \sum_{\lambda \in \widehat{G}_K}     d_\lambda \, \, f \ast \varphi_\lambda(g)  \end{equation}
where $ \widehat{G}_K$ is the subset of the unitary dual $ \widehat{G} $ consisting of equivalence classes of class-1 representations of $ G $ and for each $ \lambda \in \widehat{G}_K ,  \varphi_\lambda $ is the associated zonal spherical function and $ d_\lambda $ is the dimension of the space on which  $ \lambda $ is realised. As the spherical functions $ \varphi_\lambda $ satisfy the identity
$$  \int_K \varphi(gkh) \, dk  = \varphi_\lambda(g) \, \varphi_\lambda(h),$$
from  (\ref{peter-weyl}) we can easily prove the following expansion for the spherical means:
\begin{equation}\label{spher-mean-exp}
f(g,r )  = \sum_{\lambda \in \widehat{G}_K}     d_\lambda \, \, f \ast \varphi_\lambda(g) \,\, \varphi_\lambda(r).
\end{equation}

In the above we have written $ \varphi_\lambda(r) $ in place of $ \varphi_\lambda(a_r).$  These spherical functions are known explicitly. They are expressible in terms of Jacobi polynomials (See Helgason \cite{H}). In fact, $\varphi_\lambda(a_r)=P^{(\alpha,\beta)}_\lambda(r)$ where $P^{(\alpha,\beta)}_\lambda(r)$ are Jacobi polynomials with parameters $(\alpha,\beta)$ associated to the symmetric space $G/K$.  We recall some more properties of the spherical functions $ \varphi_\lambda.$ They are normalised so that $ \varphi_\lambda(0) = \varphi_\lambda(e) = 1 $ and  $ \{  d_\lambda^{1/2} \,
\varphi_\lambda(r), \, \lambda \in \widehat{G}_K  \} $ forms an orthonormal basis for $ L^2(A, w(r) dr).$ 
The spherical functions $ \varphi_\lambda $ are eigenfunctions of the Laplace-Beltrami operator $ \Delta $ on $ X $ with certain eigenvalues, say $ c_\lambda.$  Let $  \Delta_0 $ be the radial part of $ \Delta $ so that $  \Delta_0 \varphi_\lambda(r) = c_\lambda  \, \varphi_\lambda(r).$  Thus we see that  (\ref{spher-mean-exp}) is an expansion in terms of the functions $ \varphi_\lambda(r), \lambda \in \widehat{G}_K $ and we are in a position to apply Theorem \ref{chernoff} to prove a version of Chernoff's theorem for compact symmetric spaces.\\

In order to proceed further we need some information regarding the eigenvalues $ c_\lambda $ associated to $ \varphi_\lambda.$  Though one can describe the eigenvalues $ c_\lambda$ in terms of representation theory of $ G,$ for our purposes it is not necessary.  All compact symmetric spaces of rank one have been classified and in each case the zonal spherical functions $ \varphi_\lambda $ and the eigenvalues  $ c_\lambda $ are explicitly known.  The  spheres $ S^d \subset \R^{d+1}, d\geq 1$, the real projective spaces $ P_d(\R), 
d \geq 2,$ the complex projective spaces $ P_l, l\geq 2,$ the quaternionic projective spaces $ P_l(\mathbb{H}), l\geq 2 $ and the Cayley plane $ P_2(\text{Cay})$, see \cite{W}. In all these cases $ \widehat{G}_K $ is parameterized by non-negative integers (or even integers) and the spherical functions are given by Jacobi polynomials. For the spheres $ S^d $ the eigenvalues are given by $ c_n = n(n+d-1) $ and for $ P_d(\R) $ one has $ c_{2n} = 2n(2n+d-1).$ For the other projective spaces one has $ c_n = n(n+m+d) $ where $ d =2,4,8 $ and $ m = l-2, 2l-3, 3 $ for $ P_l(\C), P_l(\mathbb{H}) $ and $ P_2(\text{Cay}) $ respectively. For these information we refer to \cite{S} ( see also \cite{CRS}).\\

%It is customary to add a suitable constant to $ \rho_X $ and consider  $ \Delta = \widetilde{\Delta}+\rho_X .$ It is clear that by a suitable choice of $ \rho_X $ we can arrange that the eigenvalues of $ \Delta $ are squares and $ \Delta \geq \rho_X $ and all the hypothesis of Theorem \ref{chernoff} are satisfied. To be more precise, we can take $\rho_{S^d}$ or $\rho_{P_d(\mathbb{R})}:=\frac14(d-1)^2$ and $\rho_X=\frac14(m+d)^2$ where $ d =2,4,8 $ and $ m = l-2, 2l-3, 3 $ for $ X=P_l(\C), P_l(\mathbb{H}) $ and $ P_2(\text{Cay}) $ respectively.
 We now state and prove a version of Chernoff's theorem for $ \Delta$ on $ X.$
\begin{thm}\label{cher-sym} Let $ f \in C^{\infty}(G/K) $ be such that $ \Delta^mf \in L^2(G/K) $ for all $ m \geq 0 $  and satisfies the  Carleman condition 
 $ \sum_{m=1}^\infty  \| \Delta^m f \|_2^{-1/(2m)} = \infty.$ Then $ f $ cannot vanish on any open set unless it is identically zero.
\end{thm}
\begin{proof} From (\ref{spher-mean-exp}) we see that for any $ m \in \Na,$ 
\begin{equation}\label{spher-mean-norm} \int_0^R  | \Delta_0^m f(g,r)|^2  \, w(r)\, dr = \sum_{\lambda \in \widehat{G}_K}    c_\lambda^{2m} \, d_\lambda \,  | f \ast \varphi_\lambda(g)|^2 . \end{equation}
Using the identity  $ \varphi_\lambda \ast \varphi_\lambda = \varphi_\lambda ,$ satisfied by the spherical functions, we estimate
$$  d_\lambda\, | f \ast \varphi_\lambda(g)|^2  \leq   d_\lambda\,  \| f\ast \varphi_\lambda \|_2^2 \| \varphi_\lambda\|_2^2 = \| f\ast \varphi_\lambda \|_2^2 $$
and consequently (\ref{spher-mean-norm}) gives the following relation between norms of $ \Delta_0^mf(g,\cdot) $ and $ \Delta^m f(\cdot)$:
\begin{equation}\label{norm-rel} \int_0^R  | \Delta_0^m f(g,r)|^2  \, w(r)\, dr \leq  \sum_{\lambda \in \widehat{G}_K}    c_\lambda^{2m} \, d_\lambda \,  \| f \ast \varphi_\lambda\|_2^2 
= \int_G |\Delta^mf(g)|^2 dg \end{equation}
valid for all  functions on the symmetric space $ X.$ This allows us to conclude that  the sequence $ \| \Delta_0^m f(g,\cdot)\|_{L^2(A,w)} $ satisfies the Carleman condition whenever $ \| \Delta^m f\|_2 $ satisfies the same.\\

Now, suppose $ f $ vanishes on an open set $ V \subset G.$ Then for any $ g \in V $, it follows  from (\ref{spher-mean}) that $ f(g,r) $ vanishes on an interval $ (0,\delta) $ as a function of $ r.$ By Theorem \ref{chernoff} we conclude that $ f(g,r) = 0 $ for all $ r>0 $ and hence by calculating the $ \varphi_\lambda $ coefficients of $ f(g,\cdot) $  from (\ref{spher-mean-exp}) we obtain $ f \ast \varphi_\lambda(g) = 0 $ for all $ g \in V.$ The real analyticity of $ f \ast \varphi_\lambda $ forces it to be identically zero. As this is true for any $ \lambda \in \widehat{G}_K $ we conclude that $ f = 0 $ proving the theorem.
\end{proof}

\begin{rem} An exact analogue of Theorem \ref{cher-sym} for non-compact Riemannian symmetric spaces have been proved in Bhowmik-Pusti-Ray using the higher dimensional version of Theorem \ref{chernoff} due to de Jeu \cite{J}. In fact, we are inspired by the proof given in \cite{BPR}.
\end{rem}
%\begin{rem}
%	The conclusion of Theorem \ref{cher-sym} still holds true if we replace $\Delta$ by the Laplace-Beltrami operator $\widetilde{\Delta}$. As mentioned in the introduction, the modified Laplace-Beltrami operator $\Delta$ has been considered to enhance the beauty of the calculations involved.  
%\end{rem}
A close examination of the proof of Theorem \ref{cher-sym} shows that   the sequence $ \| \Delta_0^m f(g,\cdot) \|_{L^2(A,w)}^{-1/2m} $ satisfies the Carleman condition under the weaker assumption that 
$$ \sum_{m=1}^\infty \Big( \sum_{\lambda \in \widehat{G}_K}    c_\lambda^{2m} \, d_\lambda \,  | f \ast \varphi_\lambda(g)|^2 \Big)^{-1/4m} = \infty $$ 
for all $ g $ on an open subset of $ G.$ The above condition is guaranteed if we assume that the spectral projections $ f \ast \varphi_\lambda(g) $ have enough decay as a function of the eigenvalues $ c_\lambda.$ For example, as shown in \cite{BPR},  if we assume that $ |f \ast \varphi_\lambda(g)| \leq C_g e^{-t_g \sqrt{c_\lambda}} $ for some $ t_g> 0 $ then 
$ \| \Delta_0^m f(g,r)\|_{L^2(A,w)}^{-1/2m} $ will satisfy the Carleman condition. It follows that such a function cannot vanish on a nonempty open set unless it is identically zero.  Actually, as we show in the next proposition, the same conclusion can be drawn by assuming a weaker decay on the spectral projections $ f \ast \varphi_\lambda(g).$  A continuous version of the following proposition  is  implicitly available in the work of  Ingham \cite[Page no.31]{I}.  We present a complete proof for the convenience of the reader closely following \cite{I} and making necessary modifications.

\begin{prop}\label{prop-carl}  Let $ \theta $ be a positive decreasing function on $ [0, \infty) $  vanishing at infinity that satisfies  $\theta(t)\geq c(1+t)^{-1/2}$ for $t\geq 1$ and $ \int_1^\infty \theta(t) t^{-1} dt = \infty.$ Then for any $ k \in \mathbb{N} $ and $ \rho >0$, the sequence $ a_m $ defined by
$$  a_m^2  = \sum_{n=0}^\infty  (n+\rho)^{4m+k} \, e^{-2(n+\rho)\,\theta(n+\rho)} $$ satisfies the Carleman condition $ \sum_{m=1}^\infty  a_m^{-\frac{1}{2m}} = \infty.$
\end{prop}
\begin{proof}  For any $m\in \mathbb{N}$ we write $a_m^2 = I_1 + I_2 $ where   
 \begin{equation}
	I_1 = \sum_{n=0}^{m^4}(n+\rho)^{4m+k}e^{-2(n+\rho)\theta(n+\rho)},\,\,\,  I_2 = \sum_{n=m^4+1}^{\infty}(n+\rho)^{4m+k}e^{-2(n+\rho)\theta(n+\rho)} 
\end{equation}
and  estimate each sum separately. 
When $n\leq m^4$, we have $\theta(n+\rho)\geq \theta(m^4+\rho)$ which yields
$$\sum_{n=0}^{m^4}(n+\rho)^{4m+k}e^{-2(n+\rho)\theta(n+\rho)}\leq (m^4+\rho)^{k+1}e^{2\theta(m^4+\rho)}\sum_{n=0}^{m^4}(n+\rho)^{4m-1}e^{-2(n+1+\rho)\theta(m^4+\rho)}.$$  Now note that the sum on the right hand side of the above equation is dominated by
$$  \sum_{n=0}^{m^4}\int_{n}^{n+1}(x+\rho)^{4m-1}e^{-2(x+\rho)\theta(m^4+\rho)}dx  \leq  \int_{0}^{\infty}(x+\rho)^{4m-1}e^{-2(x+\rho)\theta(m^4+\rho)}dx.
$$ Now  the change of variables $y=2(x+\rho)\theta(m^4+\rho)$  transforms the above integral into  
\begin{align*}
	 \left(\frac{1}{2\theta(m^4+\rho)}\right)^{4m}\int_{0}^{\infty}y^{4m-1}e^{-y}dy=\left(\frac{1}{2\theta(m^4+\rho)}\right)^{4m} \Gamma(4m) .
\end{align*} Therefore, we have the estimate
\begin{align*}
	I_1\leq C (m^4+\rho)^{k+1}e^{2\theta(\rho)}\left(\frac{1}{2\theta(m^4+\rho)}\right)^{4m} \Gamma(4m)
\end{align*} where we have used the fact that $e^{\theta(m^4+\rho)}\leq e^{\theta(\rho)}$ which is immediate since $\theta$ is decreasing.  By making use of the Stirling formula $\Gamma(x)=\sqrt{2\pi}x^{x-1/2}e^{-x}e^{\alpha(x)/12}, \ 0<\alpha(x)<1,$ valid for $x>0,$  we  conclude that for large $m$ $$I_1\leq  \left(C\frac{4m}{\theta(2m^4)}\right)^{4m}$$ where $C$ is a constant independent of $m$.
In order to estimate $I_2$, we use the assumption on $\theta$  viz. $\theta(t)\geq c(1+t)^{-1/2}\geq 2ct^{-1/2}$ for $t\geq 1.$ Using this and following the same procedure as above, we have 
\begin{align*}
	I_2& \leq\sum_{n=m^4+1}^{\infty} \int_{n}^{n+1}(x+\rho)^{4m+k}e^{-2c(x-1+\rho)^{\frac12}}dx
	\leq \int_{m^4+\rho}^{\infty} (x+1)^{4m+k}e^{-2cx^{1/2}}dx
\end{align*}
For  $x>m^4+\rho>1$, $x+1\leq 2x,$ and hence the last integral  is dominated by 
$$e^{-c(m^4+\rho)^{\frac12}}\int_{m^2+\rho}^{\infty}x^{4m+k}e^{-cx^{1/2}}dx  \leq  2 e^{-c(m^4+\rho)^{\frac12}} \int_{0}^{\infty}y^{8m+2k+1}e^{-cy}dy. $$
Evaluating the Gamma integral, we obtain the estimate
$$I_2\leq  2\, c^{-(8m+2k+2)} \Gamma(8m+2k+2)e^{-c(m^4+\rho)^{\frac12}}\leq  2\, c^{-(8m+2k+2)} \, \Gamma(8m+2k+2) e^{-cm^2}. $$ Now using the Stirling formula mentioned above, we see that  for large $m$ there exists a constant $C$ such that   
$$   c^{-(8m+2k+2)}  \Gamma(8m+4k+2)e^{-cm^2}\leq    C \,(c^{-1} m e^{-cm/8} )^{8m}.$$
 But the right hand side of above goes to zero as $m\rightarrow \infty.$ Hence   for large values of $m$ we have proved the estimate
$$	a_m^2\leq  \left(\frac{4Cm}{\theta(2m^4)}\right)^{4m}.$$  Finally, as  $\int_{1}^{\infty}\theta(t)t^{-1}dt=\infty$ implies $\sum_{m=1}^{\infty}\frac{\theta(2m^4)}{m}=\infty$, from the above estimate we see that  $\sum_{m=1}^{\infty} a_m^{-\frac1{2m}}=\infty$ proving the proposition.
\end{proof}
The following  result, which is in the spirit of the classical Ingham's theorem, gives the admissible decay of  the spectral projections  of a function $ f $ that vanishes on an open set. 

\begin{thm}\label{ing-sym} Let $ \theta $ be a positive decreasing function on $ [0, \infty) $  vanishing at infinity that satisfies $ \int_1^\infty \theta(t) t^{-1} dt = \infty.$ Then for any function $f\in L^2(G)$ vanishing on a nonempty open set $ V \subset G $ the estimates
$$  |f \ast \varphi_\lambda(g)| \leq C_g e^{- \sqrt{c_\lambda} \, \theta(\sqrt{c_\lambda})},\,\, $$
cannot hold  for all $ g \in V $  unless the function vanishes identically.
\end{thm}
\begin{proof}  Let us begin with some observations.  From the classification of compact symmetric spaces we know that $ \widehat{G}_K $  can be identified either with $ \Na $ or $ 2\Na$.  Also by a proper choice of $ \rho_X $  we can arrange that  $ c_\lambda +\rho^2_X= (n+\rho_X)^2 .$ To be more precise, for $X=S^d \ \text{or}\  P_d(\mathbb{R})$, $\rho_X=\frac12(d-1)$ and $\rho=\frac12(m+d)$ where $ d =2,4,8 $ and $ m = l-2, 2l-3, 3 $ for $ P_l(\C), P_l(\mathbb{H}) $ and $ P_2(\text{Cay}) $ respectively. We also know that $ d_n $ has polynomial growth as a function of $ n$, see for example in \cite[Table 2, page no. 90]{S}. In fact, we can always have $d_n\leq C(n+\rho_X)^k$ for some $k$ depending on $X$.  \\

First we assume an extra condition that $ \theta(t) \geq c (1+t)^{-1/2}$ for $t\geq 1$.  In view of (\ref{spher-mean-norm} ) the hypothesis on $ f \ast \varphi_\lambda $  is estimated as
\begin{equation}\label{esti} \| \Delta_0^m f(g,\cdot)\|_{L^2(A,w)}^2 \leq C \sum_{\lambda \in \widehat{G}_K} d_\lambda \,c_\lambda^{2m} \,e^{- 2\sqrt{c_\lambda} \, \theta(\sqrt{c_\lambda})} .\end{equation}
Under the identification mentioned above, the above equation transforms into 
\begin{align}
	\label{esti2}
	 \| \Delta_0^m f(g,\cdot)\|_{L^2(A,w)}^2 \leq C \sum_{n \in\mathcal{N}} d_n \,c_n^{2m} \,e^{- 2\sqrt{c_n} \, \theta(\sqrt{c_n})}
\end{align}
where $\mathcal{N}$ denotes $\mathbb{N}$ or $2\mathbb{N}$ depending on the symmetric space $X$.
By our choice of $\rho_X$, $ n^2 \leq  c_n  \leq (n+\rho_X)^2$ and hence as $\theta$ is decreasing we get
\begin{align*}
	 \sum_{n \in\mathcal{N}} d_n \,c_n^{2m} \,e^{- 2\sqrt{c_n} \, \theta(\sqrt{c_n})}\leq C \sum_{n \in\mathcal{N}} (n+\rho_X)^{2m+k} \,e^{- 2 n \, \theta(n+\rho_X)}.
\end{align*} 
Once again, as $\theta(n+\rho_X)\leq \theta(\rho_X)$ we have 
$$ n \theta(n+\rho_X) = (n+\rho_X) \theta(n+\rho_X) - \rho_X \theta(n+\rho_X)  \geq (n+\rho_X) \theta(n+\rho_X) - \rho_X \theta(\rho_X) $$ so  that the above sum is dominated by 
\begin{align*}
	e^{ 2\,\rho_X\theta(\rho_X)}  \,  \sum_{n \in\mathcal{N}} (n+\rho_X)^{2m+k} \,e^{- 2(n+\rho_X) \, \theta(n+\rho_X)}.
\end{align*} 
Thus  from \ref{esti2} we obtain the estimate
\begin{align*}
	\| \Delta_0^m f(g,\cdot)\|_{L^2(A,w)}^2 \leq C \, \sum_{n \in\mathcal{N}} (n+\rho_X)^{2m+k} \,e^{- 2(n+\rho_X) \, \theta(n+\rho_X)}
\end{align*} 
and hence the Carleman condition $\sum_{m=1}^{\infty}\| \Delta_0^m f(g,\cdot)\|_{L^2(A,w)}^{-\frac1{2m}}=\infty $ follows from Proposition \ref{prop-carl}. Therefore, by using Theorem \ref{cher-sym} we see that  for each $g\in V$, $f(g,r)=0$ for all $r>0$. Now proceeding as in the last part of Theorem \ref{cher-sym} we conclude that $f=0.$\\

In order to remove the extra assumption on $ \theta $ we need to construct a compactly supported $ K$-invariant function $ h $ for which $ |\hat{h}(\lambda)| \leq C e^{-c\sqrt{c_{\lambda}} (1+\sqrt{c_\lambda})^{-1/2}} $ where $ \hat{h}(\lambda) = ( h, \varphi_\lambda) $ are the Jacobi (polynomial) coefficients of $ h.$ Such a construction will be done in section 6, see Theorem 6.2. More generally, we construct a family of $K$-invariant functions $ h_\delta,\, \delta >0 $ such that  $ |\hat{h}_\delta(\lambda)| \leq C e^{-\sqrt{c_{\lambda}} \theta_\delta(\sqrt{c_\lambda})}  $ where $ \theta_\delta(t) \geq a_\delta (1+t)^{-1/2} $ and $ \hat{h}_\delta(\lambda) \rightarrow 1 $ as $ \delta \rightarrow 0,$ see Remark \ref{rem-ing}.
Since  $ h_\delta $ are $K$-invariant, $ f \ast h_\delta \ast \varphi_\lambda(g) = \hat{h}_\delta(\lambda) \,  f \ast \varphi_\lambda(g)  $ it follows that $ f \ast h_\delta $ satisfies the hypothesis of the theorem with  $ \theta $ replaced  by $ \theta(t)+ \theta_\delta(t)$ we can conclude that $ f \ast h_\delta = 0 .$  This in turn implies that  $ \hat{h}_\delta(\lambda) \,  f \ast \varphi_\lambda(g)  = 0 $ and by letting $ \delta \rightarrow 0 $ we conclude that $ f \ast \varphi_\lambda = 0 $ for all $ \lambda $ and hence $ f = 0.$ This proves the theorem.
\end{proof}

\section{Laguerre polynomials and Special Hermite expansions}

In this section we consider the special Hermite operator $ L $ ( also called twisted Laplacian) and prove analogues of Theorems \ref{cher-sym} and \ref{ing-sym}. A general reference for this section is \cite[Chapter 2]{TH3}. The operator under consideration is
$$ L = -\Delta+\frac{1}{4}|z|^2 -i \sum_{j=1}^n \big( x_j \frac{\partial}{\partial y_j}- y_j \frac{\partial}{\partial x_j} \big) $$
which is an elliptic operator on $ \C^n $ with an explicit spectral decomposition. The spectrum consists of the integers of the form $ (2k+n), k \in \Na $ and the eigenspaces associated to each of these eigenvalues are infinite dimensional.  They are spanned by $ \Phi_{\alpha,\beta}, \alpha, \beta \in \Na^n, |\alpha| = k $ where $ \Phi_{\alpha,\beta} $ are the special Hermite functions defined in terms of the Hermite functions. Let us briefly recall the definition of Hermite functions.
The Hermite polynomials $H_k(x)$ for $k=0,1,2,...$ and $x\in\mathbb{R}$ are  defined by $$H_k(x):=(-1)^ke^{x^2}\frac{d^k}{dx^k}(e^{-x^2}).$$ The normalised Hermite functions on $\mathbb{R}$ are given by $$h_k(x):=(2^k\sqrt{\pi}k!)^{-\frac{1}{2}}H_k(x)e^{-\frac{1}{2}x^2}.$$
 We define the normalised multi-dimensional Hermite functions on $ \R^n $ by $$\Phi_{\alpha}(x):=\prod_{j=1}^{n}h_{\alpha_j}(x_j), ~~ x \in\R^n,\, \alpha\in\mathbb{N}^n.$$
 The special Hermite functions are then defined as the matrix coefficients of the Schrodinger representation of the Heisenberg group. More explicitly,
 $$ \Phi_{\alpha,\beta}(z) = (2\pi)^{-n/2}  \int_{\R^n}  e^{i(x \cdot \xi+ \frac{1}{2} x \cdot y)} \Phi_\alpha(\xi+y) \Phi_\beta(\xi) d\xi.$$
 These functions form an orthonormal basis for $ L^2(\C^n) $ leading to the special Hermite expansion  (See \cite[Theorem 2.3.1]{TH3})  
 $$ f(z)= \sum_{\alpha, \beta \in \Na^n} \langle f , \Phi_{\alpha,\beta} \rangle \Phi_{\alpha,\beta}(z) = \sum_{k=0}^\infty \big(   \sum_{|\alpha|=k} \sum_{\beta \in \Na^n}\langle f , \Phi_{\alpha,\beta} \rangle \Phi_{\alpha,\beta}(z)\big).$$
 
 The above can be put in a compact form in terms of the Laguerre functions $ \varphi_k^{n-1}(z).$ 
 Let $ L_k^{n-1}(t)  $ stand for  Laguerre polynomials of  type $(n-1).$ We define
 $$\varphi_k^{n-1}(z) = \frac{k!(n-1)!}{(k+n-1)!}  L_k^{n-1}(\frac{1}{2}|z|^2) e^{-\frac{1}{4}|z|^2} .$$
 It is then known that (\cite[page no. 58]{TH3}). 
 $$    \sum_{|\alpha|=k} \sum_{\beta \in \Na^n}\langle f , \Phi_{\alpha,\beta} \rangle \Phi_{\alpha,\beta}(z) = (2\pi)^{-n}  f\times \varphi_k^{n-1}(z) $$
 where for any two functions $ f $ and $ g $ from $ L^1(\C^n) $ their twisted convolution is defined by
 $$ f \times g(z) = \int_{\C^n} f(z-w) g(w) e^{\frac{i}{2} \Im (z \cdot \bar{w})} dw.$$
 Thus the special Hermite expansion of  a function $ f \in L^2(\C^n) $ and Parseval's identity reads 
 \begin{equation}\label{parseval} f(z)  = (2\pi)^{-n} \sum_{k=0}^\infty f\times \varphi_k^{n-1}(z),\, \,\,  \|f\|_2^2  = (2\pi)^{-n} \sum_{k=0}^\infty  \|f\times \varphi_k^{n-1}\|_2^2
 \end{equation}
 and each $ f \times \varphi_k^{n-1} $ is an eigenfunction of the operator $ L $ with eigenvalue $ (2k+n).$\\
 
 First we prove the following version of Chernoff's theorem for the special Hermite operator.
 
 \begin{thm}\label{cher-sp-her} Let $ f \in C^{\infty}(\C^n) $ be such that $ L^mf \in L^2(\C^n) $ for all $ m \geq 0 $  and satisfies the Carleman condition 
 $ \sum_{m=1}^\infty  \| L^m f \|_2^{-1/(2m)} = \infty.$ Then $ f $ cannot vanish on any open set unless it is identically zero.
\end{thm}
 \begin{proof}
 We apply  Theorem \ref{chernoff} to expansions in terms of  the Laguerre functions  
$$  \psi_k^{n-1}(r)  =   \frac{k!(n-1)!}{(k+n-1)!} \varphi_k^{n-1}(r) $$
which form an orthogonal basis for the space $ L^2(\R^+, r^{2n-1}dr).$   These are eigenfunctions of the Laguerre operator $ L_0 = \big(-\frac{d^2}{dr^2}- \frac{2n-1}{r} \frac{d}{dr} +\frac{1}{4} r^2 \big)$ with eigenvalues $ (2k+n).$ This follows from the fact that $ L_0 $ are the radial part of $ L $ and $ \varphi_k $ are radial eigenfunctions of $ L.$ Moreover, it is also known that
$\psi_k^{n-1}(0) = 1 $ and
\begin{equation}\label{normalising}  c_k  =  \int_0^\infty  \psi_k^{n-1}(r)^2 r^{2n-1} dr =  2^{n-1}(n-1)! \frac{k!(n-1)!}{(k+n-1)!}   .\end{equation}

 In order to apply Theorem \ref{chernoff}  we make use of the twisted spherical mean value operator  
 $$ f \times \mu_r(z) =  \int_{|w|=r}  f(z-w) e^{\frac{i}{2}\Im( z\cdot \bar{w})} d\mu_r $$ 
 where $ \mu_r $ is the normalised surface measure on $ |w| = r $ in $ \C^n.$  The important fact about the twisted spherical means which is relevant for us  is that it  has the following expansion in terms of Laguerre functions $ \varphi_k^{n-1}(r):$ 
\begin{equation}\label{twisted} f \times \mu_r(z)  = (2\pi)^{-n} \sum_{k=0}^\infty  f \times \varphi_k^{n-1}(z)  \frac{k!(n-1)!}{(k+n-1)!} \varphi_k^{n-1}(r).\end{equation}
This has been proved in \cite[Theorem 4.1]{TH0}. 
Suppose now $ f $ vanishes on an open set $ V.$  It then follows from the definition that  for any $ z \in V $ the function $ F_z(r)  = f \times \mu_r(z) $ vanishes for $ 0 < r <\delta(z) $  where $ \delta(z) $ is the distance between $ z $ and the complement of $ V.$ Since $ \psi_k^{n-1}(r) $ are eigenfunctions of $ L_0, $ the expansion of the  function $ F_z(r) = f \times \mu_r(z) $ gives 
$$  \int_0^\infty |L_0^mF_z(r)|^2 r^{2n-1} dr = (2\pi)^{-2n} \sum_{k=0}^\infty   (2k+n)^{2m} \,\, c_k\,  | f\times \varphi_k^{n-1}(z)|^2 \,   . $$
Since $ f \rightarrow (2\pi)^{-n} f \times \varphi_k^{n-1} $ is a projection, we have 
$$ f \times \varphi_k^{n-1} = (2\pi)^{-n} f \times \varphi_k^{n-1} \times \varphi_k^{n-1} $$ and hence using $ \| f \times g \|_2 \leq \|f\|_2 \|g\|_2 ,$ we get the estimate
$$  | f\times \varphi_k^{n-1}(z)|^2 \leq (2\pi)^{-2n}    \| f \times \varphi_k^{n-1} \|_2^2 \| \varphi_k^{n-1}\|_2^2 = C_n \frac{(k+n-1)!}{k!(n-1)!}     \| f \times \varphi_k^{n-1} \|_2^2.    $$
Therefore, using (\ref{parseval}) applied to the function $ L^m f $  and (\ref{normalising}) we obtain 
$$  \| L_0^m F_z \|_2^2 \leq C_n\, \sum_{k=0}^\infty    \, (2k+n)^{2m} \, | f\times \varphi_k^{n-1}(z)|^2 \,  \leq C_n \| L^m f\|_2^2 . $$
From the above it follows that $ \sum_{m=0}^\infty  \| L_0^m F_z\|_2^{-1/2m} \geq  \sum_{m=0}^\infty  \| L^m f \|_2^{-1/2m} $ and therefore by Theorem \ref{chernoff} we conclude that $ f \times \mu_r(z) = 0 $ for all $ r > 0$ under the condition that $ z \in V.$

Being the Laguerre coefficient of $ F_z(r)$ which is identically zero,  $ f \times \varphi_k^{n-1}(z) = 0 $ on the open set $ V.$ But $ f \times \varphi_k^{n-1} $ is real analytic as it is an eigenfunction of the analytic hypo-elliptic operator $ L $ and hence $ f \times \varphi_k^{n-1}(z) = 0 $ for all $ z \in \C^n.$ As this is true for any $ k $ we conclude that $ f = 0.$
\end{proof}

\begin{rem} As in the case of compact symmetric spaces we remark that the above theorem is true under the weaker assumption that
$$ \sum_{m=1}^\infty \Big( \sum_{k=0}^\infty     (2k+n)^{2m} \,  c_k  \,  | f \times  \varphi_k^{n-1}(z)|^2 \Big)^{-1/4m} = \infty $$
for all $ z \in V $ where $ f $ is assumed to vanish.
\end{rem}

We also have the following version of Ingham's theorem.

\begin{thm}\label{ing-sp-her} Let $ \theta $ be a positive decreasing function on $ [0, \infty) $  vanishing at infinity that satisfies $ \int_1^\infty \theta(t) t^{-1} dt = \infty.$ For any  function $ f \in L^2(\C^n)$  vanishing on a nonempty open set $ V \subset \C^n $ the estimates
$$  |f \times \varphi_k^{n-1}(z)| \leq C_z e^{- \sqrt{2k+n} \, \theta(\sqrt{2k+n})},\,\,  $$
cannot hold for all  $ z \in V $ unless the function vanishes identically.
\end{thm}
\begin{proof} Proceeding as in the case of symmetric spaces, the theorem is first proved under the extra assumption $ \theta(t) \geq c(1+t)^{-1/2} $ for $t\geq 1$ on $ \theta.$ Then we make use of the construction mentioned above. For any $ \delta > 0 $ we can construct  a radial $ g_\delta $ supported on $ B(0,\delta) $ having the following two properties: (i)   for any $ f \in L^1(\C^n) ,\,\,f \times g_\delta \rightarrow f $  in  $ L^1(\C^n) $ as $ \delta \rightarrow 0 $ and (ii) for any $ k \in \Na$,
$$ \| g_\delta \times  \varphi_k^{n-1}\|_2 \leq C e^{-\sqrt{2k+n} \,\, \theta_\delta(\sqrt{2k+n})}  $$  where $ \theta_\delta(t) \geq a_\delta (1+t)^{-1/2} .$
We refer to the next section for the construction of such a family $ g_\delta,$ see Theorem \ref{splhermite-ingham}.
 Then the function $ f \times g_\delta $ satisfies the hypothesis on a smaller open subset of $ V $ and $ \theta $ replaced by $ \theta(t)+\theta_\delta(t) \ge a_{\delta} (1+t)^{-1/2}.$ Hence by the previous part of the theorem we can conclude $ f \times g_\delta = 0 $ which will then prove $ f = 0$ as $ f \times g_\delta \rightarrow f .$ 
\end{proof}

\section{Hermite expansions}

 In this section we consider the Hermite operator $ H = -\Delta+|x|^2 $ on $ \R^n $ whose eigenfunctions are given by Hermite functions.   It is well known that these functions are eigenfunctions of  $H=-\Delta+|x|^2$ with eigenvalues $ (2|\alpha|+n)$ and  $\{\Phi_{\alpha}:\alpha\in\mathbb{N}^n\}$ forms an orthonormal basis for $L^2(\R^n).$ So every $f\in L^2(\R^n)$ has the expansion
 $$f=\sum_{\alpha \in \Na^n}(f,\Phi_{\alpha})\Phi_{\alpha}$$  where the Hermite coefficient of the function $f$ are  defined by $(f,\Phi_{\alpha})=\int_{\R^n}f(x)\Phi_{\alpha}(x)dx.$ The Plancherel formula takes the form $$\|f\|_2^2=\sum_{\alpha \in \Na^n}|(f,\Phi_{\alpha})|^2.$$ 
Let $ P_k $ stand for the orthogonal projection  of $ L^2(\R^n) $ onto the $k$-th eigenspace of $ H $ spanned by $ \Phi_\alpha, |\alpha| =k $ we can write the Hermite expansion as 
$$ f = \sum_{k=0}^\infty P_kf ,\,\,\,   H f = \sum_{k=0}^\infty (2k+n) P_kf.$$

The connection between Hermite expansions and Laguerre functions is brought out by considering the Weyl transform $W$  of the surface measure $ \mu_r $ on $ S_r^{2n-1} $ in $ \C^n.$ By using the fact that $ (2\pi)^{-n}  W(\varphi_k) = P_k$ it follows from the expansion (\ref{twisted}) that we have
\begin{equation}\label{her-lag}  W(\mu_r) = \sum_{k=0}^\infty  \frac{k!(n-1)!}{(k+n-1)!} \,  \varphi_k^{n-1}(r)  \, P_k.\end{equation}
The action of $ W(\mu_r) $ on a function $ f \in L^2(\R^n) $ is by definition
$$ W(\mu_r)f(\xi)  = \int_{|z| = r} \pi(z)f(\xi) d\mu_r(z) = \int_{ |z|=r}  e^{i(x\cdot \xi+ \frac{1}{2} x \cdot y)} f(\xi+y) d\mu_r .$$
From the above it is clear that $ W(\mu_r)f(\xi) $ vanishes  over $ 0 < r < \delta-|\xi|,$  whenever $ f $ vanishes over the ball $ B(\xi,\delta).$ When $ f $ is  such a function, then from (\ref{her-lag}) we have the following expansion for $ F_\xi(r):=W(\mu_r)f(\xi)$:
\begin{equation}\label{her-lag-1}
 F_\xi(r)  = \sum_{k=0}^\infty  \frac{k!(n-1)!}{(k+n-1)!} \, P_kf(\xi)  \,  \varphi_k^{n-1}(r). \end{equation}
 Under the assumption that $ H^m f \in L^2(\R^n) $ for all $ m \in \Na $ it follows from the above that $ L_0 F_\xi(r) = W(\mu_r)(H^mf)(\xi) $ where $L_0$ is the Laguerre operator given by $$ L_0 = -\frac{d^2}{dr^2}- \frac{2n-1}{r} \frac{d}{dr} +\frac{1}{4} r^2 .$$ Hence
 \begin{equation}\label{eq-1} \int_{0}^\infty |L_0^mF_\xi(r)|^2 r^{2n-1} dr = c_n \sum_{k=0}^\infty  (2k+n)^{2m}  \frac{k!(n-1)!}{(k+n-1)!}  | P_kf(\xi)|^2 .\end{equation}
 This allows us to prove the following analogue of Chernoff's theorem for the Hermite operator.
 
 \begin{thm}\label{her-cher} Assume that $ f \in C^{\infty}(\mathbb{R}^n) $ is such that $ H^m f \in L^2(\R^n) $ for all $ m \geq 0 $ and  satisfies the condition $ \sum_{m=1}^\infty  \| H^m f \|_2^{-1/2m} = \infty.$ Then $ f $ cannot vanish in a neighbourhood of zero unless it vanishes identically.
 \end{thm}
 \begin{proof}  Fix $ \xi $ in the neighbourhood where $ f $ vanishes. We first check that $ \| L_0^mF_{\xi} \|_2^{-1/2m} $ also satisfies the Carleman condition. As in the case of special Hermite expansions, we have
 $$ P_kf(\xi) = \int_{\R^n} \Phi_k(\xi,\eta) P_kf(\eta) d\eta$$
 where $ \Phi_k(\xi,\eta) = \sum_{|\alpha|=k} \Phi_\alpha(\xi) \Phi_\alpha(\eta) $ is the kernel of $ P_k.$ From the above we obtain
 $$ |P_kf(\xi)|^2 \leq \|P_kf\|_2^2 \int_{\R^n} \Phi_k(\xi,\eta)^2 d\eta = \|P_kf\|_2^2 \,\,  \Phi_k(\xi,\xi).$$
 Good estimates for $ \Phi_k(\xi,\xi) $ are known: for example, when $ n \geq 2,$ one has the uniform estimate $\Phi_k(\xi,\xi) \leq C (2k+n)^{n/2-1} $ (see Lemma 3.2.2 in \cite{TH1}) and when $ n=1, \Phi_k(\xi,\xi) = h_k(\xi)^2 \leq C (2k+1)^{-1/6}.$ In any case,  $ \frac{k!(n-1)!}{(k+n-1)!} \Phi_k(\xi,\xi) $ is uniformly bounded and so from (\ref{eq-1}) we obtain the estimate 
\begin{equation}\label{eq-2} \| L_0^mF_\xi\|_2^2 \leq  c_n \sum_{k=0}^\infty  (2k+n)^{2m}   \| P_kf \|_2^2  = c_n \| H^mf\|_2^2  .\end{equation}
 This proves our claim that $ \| L_0^mF_{\xi} \|_2^{-1/2m} $ satisfies the Carleman condition and hence by Theorem \ref{chernoff} we conclude that $ F_\xi(r) = 0 $ for all $ r >0 $ and $ \xi $ in a neighbourhood of zero.  As before this allows us to conclude that $ P_kf $ vanishes near zero and hence by the analyticity argument (which is easy to see now as $ e^{\frac{1}{2}|\xi|^2} P_kf(\xi) $ is a polynomial) that $ P_kf = 0 $ for every $ k .$ This completes the proof.
 \end{proof}

 \begin{rem} Once again  we remark that the above theorem is true under the weaker assumption  that for all $ \xi  \in V $ where $ f $ is assumed to vanish, $$ \sum_{m=1}^\infty \Big( \sum_{k=0}^\infty     (2k+n)^{2m} \,  \frac{k!(n-1)!}{(k+n-1)!}  \,  | P_kf(\xi)|^2 \Big)^{-1/4m} = \infty .$$

 We also have the following version of Ingham's theorem.

\end{rem}

\begin{thm}\label{ing-her} Let $ \theta $ be a positive decreasing function on $ [0, \infty) $  vanishing at infinity that satisfies $ \int_1^\infty \theta(t) t^{-1} dt = \infty.$  Then for any  function $ f\in L^2(\R^n) $   vanishing on a nonempty open set $ V \subset \R^n $ the estimates
$$  | P_kf(\xi)| \leq C_\xi e^{- \sqrt{2k+n} \, \theta(\sqrt{2k+n})},\,\,  $$
cannot hold for all  $ \xi \in V $ unless the function vanishes identically.
\end{thm}
\begin{proof} Once again we first assume that $ \theta(t) \geq c (1+t)^{-1/2} $ in proving the theorem. For the general case, we make use  of the same functions $ g_\delta$ used in the proof of Theorem \ref{ing-sp-her}. We define $ f_\delta(\xi) = W(g_\delta)f(\xi),$ or more explicitly,
$$ f_\delta(\xi) =  \int_{\C^n}  g_{\delta}(x+iy) e^{i(x \cdot \xi+\frac{1}{2} x \cdot y)} f(\xi+y) dx dy.$$
As $ g_\delta $ is supported on $ |z| \leq \delta $ and $ f $ vanishes on $ V $ it follows that  for small enough $ \delta $ the function $ f_\delta $ vanishes on an open subset of $ V.$ Moreover, as $ g_\delta $ is radial,
$$ W(g_\delta)  = \sum_{k=0}^\infty  R_k(g_\delta) P_k  $$ where $ R_k(g_\delta) $ are defined by the relation $ g_\delta \times \varphi_k^{n-1} = R_k(g_\delta) \varphi_k^{n-1}.$ Consequently, we have
$$  P_k f_\delta(\xi) = P_k W(g_\delta)f(\xi) = R_k(g_\delta) P_kf(\xi).$$ The decay assumptions on $ \| g_\delta \times \varphi_k^{n-1}\| $ and on $ |P_kf(\xi)| $ allows us to conclude that 
$$ |P_kf_\delta(\xi)| \leq C e^{-\sqrt{2k+n} \, (\theta+\theta_\delta)(\sqrt{2k+n})}  $$ 
where $ \theta_\delta(t)+ \theta(t)\ge a_{\delta  }(1+t)^{-1/2}.$ Hence we can conclude that $ f_\delta = 0$ for any $ \delta >0.$ Since $ g_\delta $ is an approximate identity for $ L^1(\C^n) $ under twisted convolution, it follows that $ W(g_\delta) \rightarrow I $ on $ L^2(\R^n) $ and hence $ f_\delta \rightarrow f $ in $ L^2(\R^n).$ This proves that $ f =0.$

\end{proof}
  
  Replacing the function $ t\, \theta(t) $ by an increasing function $ \psi(t) $   Levinson \cite{Levinson} proved the following version of Ingham type  theorem:
 	\begin{thm}[Levinson]
 		Let $\psi:[0,\infty)\rightarrow [0,\infty)$ be an increasing function satisfying $\lim_{t\rightarrow \infty}\psi(t)=\infty.$ Assume that $\int_{1}^{\infty}\psi(t)t^{-2}dt=\infty.$ Suppose $f\in L^1(\R)$ and the Fourier transform of $f$ satisfies 
 		$$|\hat{f}(\xi)|\leq C e^{-\psi(\xi)}, ~\text{for all}~|\xi|\geq 1.$$ Then $f$ cannot vanish on a non-empty open set unless it is identically zero. 
 	\end{thm} 
 	We prove an analogue of the above theorem for the Hermite expansion. 
 The proof does not use the above analysis: we make use of asymptotic properties of Laguerre functions and deduce the result from Levinson's theorem for the Fourier transform on $ \R^n$, proved in \cite{BSR}.   We remark that in \cite{BR}, the authors proved the analogue of Levinson's theorem in a more general setting, namely on Riemannian symmetric spaces of noncompact type.  
 
 \begin{thm}\label{ing-her-1}  Let $ \psi $ be a positive increasing function on $ [0, \infty) $ satisfying $\lim_{t\rightarrow \infty}\psi(t)=\infty.$ Assume that $\int_{1}^{\infty}\psi(t)t^{-2}dt=\infty.$   For any  function  $ f\in L^1 \cap L^2(\R^n) $   vanishing on a nonempty open set $ V \subset \R^n $ the estimates
 	$$  \|P_k f\|_2 \leq C e^{-  \psi(\sqrt{2k+n})}\,\,  $$
 	cannot hold for all  $ k \in \Na $ unless the function vanishes identically. 
 	\end{thm} 
 	We prove the above theorem by reducing it to the case of  Fourier transform on $ \R^n.$ Thus our goal is to show  that the function $ \hat{f} $ satisfies the estimate 
 	\begin{equation}\label{four-est} |\hat{f}(\xi)|\leq C e^{-\frac{1}{4}\psi(\frac{1}{\sqrt{2}}|\xi|)  }.
 	\end{equation}
 	In order to do so , we make use of  the pointwise  estimates on the kernel of $ P_k$ which is given by
 $$  \Phi_k(x,y) =  \sum_{|\alpha|=k}  \Phi_\alpha(x)  \Phi_\alpha(y).$$  This kernel can be expressed in terms of Laguerre polynomials $ L_k^{n/2-1} $ of type $ (n/2-1) $ as shown in the following lemma, see \cite{T5} for a proof.
 
 \begin{lem}\label{formula}
 	$$\Phi_{k}(x,y)=\pi^{-\frac{n}{2}}\sum_{j=0}^{k}(-1)^jL_j^{\frac{n}{2}-1}(\frac{1}{2}|x+y|^2)e^{-\frac{1}{4}|x+y|^2}L_{k-j}^{\frac{n}{2}-1}(\frac{1}{2}|x-y|^2)e^{-\frac{1}{4}|x-y|^2}.$$
 \end{lem}
 
 Now we recall some definitions and facts about Laguerre functions. For any $\delta>-1,$ the Laguerre polynomials of type $\delta$ are defined by 
 $$e^{-t}t^{\delta}L_k^{\delta}(t)=\frac{1}{k!}\frac{d^k}{dt^k}(e^{-t}t^{k+\delta})$$
 for $t>0$ and $k \in \Na.$ The explicit form of $L_k^{\delta}(t)$ which is a polynomial of degree $k$, is given by 
 $$L_k^{\delta}(t)=\sum_{j=0}^{k}\frac{\Gamma(k+\delta+1)}{\Gamma(j+\delta+1)\Gamma(k-j+1)}\frac{(-t)^j}{j!}.$$  
 We now introduce the normalised Laguerre functions $ \mathcal{L}_k^\delta $ defined as follows.
 $$\mathcal{L}^{\delta}_k(t)=\left(\frac{\Gamma(k+1)}{\Gamma(k+1+\delta)}\right)^{\frac{1}{2}}e^{-\frac{t}{2}}t^{\frac{\delta}{2}}L^{\delta}_k(t),\:\, t >0.$$
 Then it is well known that for any fixed $\delta>-1$, $\left\lbrace\mathcal{L}^{\delta}_k\right\rbrace_{k=0}^{\infty}$ is an orthonormal basis for  $L^2(\mathbb{R}^+,dt).$  What we need are the  following estimates on these Laguerre functions which can be found in \cite{TH1}.
 \begin{lem}
 	\label{lag}
 	For $\delta>-1$, the following estimates hold.
 	$$|\mathcal{L}^{\delta}_k(x)|\leq C\begin{cases}
 	(x\nu)^{\frac{\delta}{2}} &if\ 0\leq x\leq\frac{1}{\nu}\\
 	(x\nu)^{-\frac{1}{4}} & if\ \frac{1}{\nu}\leq x\leq \frac{\nu}{2}\\
 	\nu^{-\frac{1}{4}}(\nu^{\frac{1}{3}}+|\nu-x|)^{-\frac{1}{4}}  &if\ \frac{\nu}{2}\leq x\leq \frac{3\nu}{2}\\
 	e^{-\gamma x} &if \ x\geq \frac{3\nu}{2}
 	\end{cases}$$ where $\gamma$ is a fixed constant and $\nu=2(2k+\delta+1).$
 \end{lem}
 
 We now make use of the above lemma in getting the following estimates on the kernel $ \Phi_k(x,y) $ in view of Lemma \ref{formula}.
 
 \begin{lem}
 	\label{phik}
 	For $|x|^2> 2(2k+n)$, $\Phi_{k}(x,x)\leq C \,(2k+n)^{\frac{n}{2}}\, e^{-2\gamma |x|^2}$  where  $\gamma >0$ is the same constant appearing in the above lemma.
 \end{lem}
 \begin{proof}
 	From the explicit expression of the Laguerre polynomial given above note that $$L^{\delta}_k(0)=\frac{\Gamma(k+\delta+1)}{\Gamma(k+1)\Gamma(\delta+1)}.$$ Using this in Lemma \ref{formula} we have
 	\begin{equation}\label{eqn}
 	\Phi_{k}(x,x)= \pi^{-\frac{n}{2}}\sum_{j=0}^{k}(-1)^jL_j^{\frac{n}{2}-1}(2|x|^2)e^{-|x|^2}\frac{\Gamma(k-j+\frac{n}{2})}{\Gamma(k-j+1)\Gamma(\frac{n}{2})}.
 	\end{equation}
 	Now using the expression for Laguerre function defined above note that
 	\begin{equation}
 	L_j^{\frac{n}{2}-1}(2|x|^2)e^{-|x|^2}=\mathcal{L}^{\frac{n}{2}-1}_j(2|x|^2)\Bigg(\frac{\Gamma(j+\frac{n}{2})}{\Gamma(j+1)}\Bigg)^{\frac{1}{2}}(2|x|^2)^{-\frac{1}{2}(\frac{n}{2}-1)}
 	\end{equation}
 	Using the estimate $\frac{\Gamma(k-j+\frac{n}{2})}{\Gamma(k-j+1)}\leq C (k-j)^{\frac{n}{2}-1}$ we have $$|L_j^{\frac{n}{2}-1}(2|x|^2)e^{-|x|^2}|\leq C |\mathcal{L}^{\frac{n}{2}-1}_j(2|x|^2)|(2|x|^2)^{-\frac{1}{2}(\frac{n}{2}-1)} j^{\frac{1}{2}(\frac{n}{2}-1)} (k-j)^{\frac{n}{2}-1}.$$
 	But since $|x|^2> 2(2k+n)$, by Lemma \ref{lag} we have 
 	$$|L_j^{\frac{n}{2}-1}(2|x|^2)e^{-|x|^2}|\leq C e^{-2\gamma |x|^2}(4(2k+n))^{-\frac{1}{2}(\frac{n}{2}-1)} j^{\frac{1}{2}(\frac{n}{2}-1)} (k-j)^{\frac{n}{2}-1}.$$ Now from \ref{eqn}, using the fact that $j,(k-j)\leq k$ we have 
 	$$ \Phi_{k}(x,x)\leq  C e^{-2\gamma |x|^2} (2k+n)^{-\frac{1}{2}(\frac{n}{2}-1)}k^{\frac{n}{2}-1+\frac{n}{4}-\frac{1}{2}}(k+1)
 	\leq   C (2k+n)^{\frac{n}{2}}e^{-2\gamma |x|^2}.$$ 
 \end{proof}
 
 We are now in a position to prove the estimate (\ref{four-est}) on the Fourier transform of $ f.$
 
 \begin{prop}\label{four-esti}   Assume that $ \psi $ satisfies the hypothesis of  Theorem \ref{ing-her-1}.  If we further assume that  $\psi(t)\geq  c\, t^{1/2}$ for $ t \geq 1$, then we have the estimate
 		$$  |\hat{f}(\xi)|\leq C e^{-\frac{1}{4}\psi(\frac{1}{\sqrt{2}}|\xi|) }.$$
 \end{prop}
 \begin{proof}  By defining $ g $ by the relation $f = e^{-\frac{1}{2}\psi(\sqrt{H})}g $ we note that $g\in L^2(\mathbb{R}^n).$   Indeed,  from the definition of $g$ we see that 
$\|P_kg\|_2^2 =  e^{\psi(\sqrt{2k+n})}\|P_kf\|_2^2.$ But since $\|P_k f\|_2 \leq C e^{-  \psi(\sqrt{2k+n})}$, and $\psi(t)\geq c\, t^{1/2},  t\geq 1$ we have 
 $$\|g\|_2^2\leq C\sum_{k=0}^{\infty}e^{-  \psi(\sqrt{2k+n})}\leq C\sum_{k=0}^{\infty}e^{-c(2k+n)^{\frac14}}  < \infty $$  which proves that  $g\in L^2(\mathbb{R}^n).$     Now we have 
 		$$  f(x) = \sum_{k=0}^\infty  e^{-\frac{1}{2} \psi(\sqrt{2k+n}) }\, P_kg(x) .$$
 As $ P_k $ are projections, $ P_kg(x) = P_k^2 g(x) $ and hence
 		$$ |P_kg(x)| = |\int_{\R^n} \Phi_k(x,y) P_kg(y) dy | \leq \|P_kg\|_2 \, (\Phi_k(x,x))^{1/2}.$$
 Using this in the above expansion of $ f(x) $ we  see that
 		$$ |f(x)| \leq \|g\|_2 \Big( \sum_{k=0}^\infty  e^{-  \psi(\sqrt{2k+n}) }\, \, \Phi_k(x,x) \Big)^{1/2}  .$$
 The sum taken over those $ k $ for which $ (2k+n) < \frac{1}{2} |x|^2 $ gives a good estimate in view of Lemma \ref{phik}. Indeed,
 		$$ \sum_{(2k+n) < \frac{1}{2}|x|^2}^\infty  e^{-  \psi(\sqrt{2k+n}) }\, \, \Phi_k(x,x) \leq C e^{-2\gamma |x|^2}.$$
For the remaining sum we note that $ \sqrt{2k+n} \geq \frac{1}{\sqrt{2}}|x| $ and hence as $ \psi(t) $ is assumed to be increasing we have 
 		$$\psi(\frac{1}{\sqrt{2}}|x|)  \leq \psi(\sqrt{(2k+n)})  .$$
 Consequently, using he trivial estimate $ \Phi_k(x,x) \leq c(2k+n)^{n-1} $ we have 
 		$$ \sum_{(2k+n)\geq \frac{1}{2}|x|^2}^\infty  e^{-  \psi(\sqrt{2k+n}) }\, \, \Phi_k(x,x) \leq C e^{-\frac{1}{2}\psi(\frac{1}{\sqrt{2}}|x|)  } .$$
 Thus we have proved the stated estimate for $ f(x).$ Since $ P_kf $ are eigenfunctions of the Fourier transform, $ \| P_k\hat{f}\|_2 $ also satisfies the same decay estimate 
 as $ \|P_kf\|_2 $ and hence we obtain the same estimate on $ \hat{f}(\xi).$ This completes the proof.
 	\end{proof}

	\textbf{Proof of Theorem \ref{ing-her-1}:}
 	Under the  hypothesis of Theorem \ref{ing-her-1} along with the assumption that $\psi(t)\geq c t^{1/2}$ for $t\geq 1$ we have shown that $ \hat{f} $ has the decay stated in Proposition \ref{four-esti}.  Hence  applying Levinson's theorem for the Fourier transform on $ \R^n $ proved in \cite[Theorem 2.6]{BSR}, we obtain $f=0$.

   For the general case, we proceed as in the proof of Theorem \ref{ing-her}.   Let $\theta(t)=(1+t)^{-\frac12}$. Clearly, $\theta$ is a decreasing function which vanishes at infinity and  $\int_{1}^{\infty}\theta(t)t^{-1}dt  < \infty $. Then by Theorem \ref{splhermite-ingham}, for each $\delta>0$ there is a radial function $g_{\delta}$ on $\mathbb{C}^n$ which satisfies 
 	$$\|g_{\delta}\times \varphi^{n-1}_k\|_{L^2(\mathbb{C}^n)}\leq C e^{-\psi_{\delta}(\sqrt{2k+n})}$$ where $\psi_{\delta}(t)=t\delta(1+\delta t)^{-1/2}.$ Now it is easy to see that $\psi_{\delta}$ is an increasing function and  for all $t\geq 1$, $\psi_{\delta}(t)\geq a_{\delta}t^{1/2}$ where $a_{\delta}=\sqrt{2}\delta(1+\delta)^{-1/2}.$ For every $\delta>0$, we consider the function $f_{\delta}$ defined by $ f_\delta(\xi) = W(g_\delta)f(\xi).$ Then as observed in the proof of Theorem \ref{ing-her}, for small enough $\delta, \,  f_{\delta}$ vanishes on an open set $V$ and $P_kf(\xi)=R_k(g_{\delta})P_kf(\xi)$ which yields $$\|P_kf_{\delta}\|_2= |R_k(g_{\delta})|\|P_kf\|_2.$$ Now by the hypothesis of the theorem we have 
 	$$\|P_kf_{\delta}\|_2\leq e^{-(\psi+\psi_{\delta})(\sqrt{2k+n})}, \ \forall k.$$
 		But by construction, $\psi+\psi_{\delta}$ is increasing and satisfies $\psi(t)+\psi_{\delta}(t)\geq a_{\delta}t^{1/2}$ for $t>1.$ Hence from the first part of the proof, we conclude that $f_{\delta}=0$. Now by our construction (see Theorem \ref{splhermite-ingham}) $ g_\delta $ is an approximate identity for $ L^1(\C^n) $ under twisted convolution and hence  $ W(g_\delta) \rightarrow I $ on $ L^2(\R^n) $ which implies that $ f_\delta \rightarrow f $ in $ L^2(\R^n).$ Therefore, $f=0$. 
 		\qed

 \section{Construction of  examples with prescribed  spectral decay }
   In this section we indicate how to construct compactly supported functions $ g $ on $ \Omega $ with a prescribed decay on the norms of their spectral projections. We first consider the case of the  compact Riemannian symmetric space $\Omega=G/K$. In this case we construct a $K$- invariant smooth function on  $G/K$ whose Fourier coefficients have Ingham type decay.  In order to do so, we need to recall some results from Jacobi analysis. 
  
  Let $\alpha,\beta \in \mathbb{C}$ with $-\alpha\notin \mathbb{N} $ and let $n$ be a non-negative integer. We consider the polynomials $R^{(\alpha,\beta)}_n(x):=P^{(\alpha,\beta)}_n(1)^{-1}P^{(\alpha,\beta)}_n(x)$ where $P^{(\alpha,\beta)}_n(x)$ are  Jacobi polynomials of type $ (\alpha, \beta).$ For $\alpha>\beta>-\frac12$, the Fourier-Jacobi coefficients of type $(\alpha,\beta)$ of  an even  function $ f $  on $(-\pi,\pi)$  are given by 
  	$$\tilde{f}(m):=\frac{1}{\Gamma(\alpha+1)}\int_0^{\pi}f(s)R^{(\alpha,\beta)}_n(\cos s)(\sin\frac12s)^{2\alpha+1}(\cos\frac12s)^{2\beta+1}ds,~m=0,1,2,....$$  
 Let $\mathcal{H}$ be the class of even entire functions of exponential type on $ \C $ which  are  rapidly decreasing. More precisely, we say that  a function $g$ on $\C$ belongs to $\mathcal{H}$ if $g$  is an even entire function and there are positive constants $A$ and $C_m$ such that for all $z\in\C$ and for all $m=0,1,2,...$, $g$ satisfies the following: 
$$|g(z)|\leq C_m(1+|z|)^{-m}e^{A|\Im(z )|}.$$  
	We require the following Paley-Wiener type theorem for the Fourier-Jacobi coefficients proved in Koornwinder \cite{K}
	\begin{thm}
		\label{pw}
		Let $\alpha>\beta>-\frac12$. The function $\tilde{f}$ is the Fourier-Jacobi transform of an even, compactly supported, smooth function on $(-\pi,\pi)$ if and only if there exists $g \in \mathcal{H}$ such that $A<\pi$ and $\tilde{f}(m)=g(m+(\alpha+\beta+1)/2),~m\geq 0.$
	\end{thm}
Now we use this theorem to construct a $K$-invariant function on the compact symmetric space whose spherical transform will have Ingham type decay.
 \begin{thm}
 	 Assume that $ \theta$ is a positive decreasing function on $[0,\infty)$, vanishing at infinity for which $ \int_1^\infty \theta(t) \, t^{-1} dt < \infty.$ Then there exists a $K$-invariant function  $h\in C^{\infty} (G/K)$ supported in a small neighbourhood of identity which satisfies  $|\hat{h}(\lambda)|\leq Ce^{- \sqrt{c_\lambda} \, \theta(\sqrt{c_\lambda})}$ for all $\lambda\in \widehat{G}_K .$  
   \end{thm}
\begin{proof}
Since $ \int_1^\infty \theta(t) \, t^{-1} dt < \infty$, an easy calculation yields $\int_1^\infty \psi(t) \, t^{-1} dt < \infty$ where $\psi(t):=\theta(\sqrt{t}),~t\geq 1.$ So, by Ingham's theorem (see \cite{I}) there exist an even  smooth function $f$ on $\R$,  compactly supported inside $(-\pi,\pi)$ whose Fourier transform satisfies
\begin{equation}
 \label{cs-1}
 |\hat{f}(\xi)|\leq Ce^{-|\xi|\psi(|\xi|)}\leq C e^{-|\xi|\theta(\sqrt{|\xi|}) } ,~\forall \xi.
\end{equation}   
Also in view of the  classical Paley-Wiener theorem for the Fourier transform (see \cite{stein}), it follows that $\hat{f}\in \mathcal{H}.$ Hence by Theorem \ref{pw}, there exist an even function $h\in C_0^{\infty}(-\pi,\pi)$ such that the Fourier-Jacobi transform $\tilde{h}(m)=\hat{f}(m+(\alpha+\beta+1)/2).$ But  this $h$ can be considered as a smooth $K$-invariant function on $G/K$ which is supported in a small neighbourhood of the identity. The  spherical transform of this function on $ G/K $ are  given precisely by the Fourier-Jacobi  coefficients of $h $ of type $(\alpha,\beta)$ where $(\alpha,\beta)$ is associated to the symmetric space $G/K.$ Now as mentioned in Section 3, $\widehat{G}_K$ can be identified with either $\mathbb{N}$ or $2\mathbb{N}$ and in all cases $c_{\lambda}\geq 1$ for all $\lambda\in \widehat{G}_K$. Also we know that the inequality in \ref{cs-1} is translation invariant. Hence we can always arrange that for all $\lambda\in \widehat{G}_K$, $|\tilde{h}(\lambda)|\leq Ce^{- \sqrt{c_\lambda} \, \theta(\sqrt{c_\lambda})}$ proving the theorem.
\end{proof}	

\begin{rem}\label{rem-ing}Once we have constructed $ f $ satisfying \ref{cs-1}, we can normalize it so that $ \| f\|_1 =1 .$ For each $ \delta > 0 $ the function $ f_\delta(x) = \delta^{-1} f(\delta^{-1} x) $ is supported in $ |x| \leq A \delta $ whose Fourier transform has the decay $  |\hat{f_\delta}(\xi)|\leq Ce^{-\delta |\xi|\psi(\delta |\xi|)}.$ As $ f_\delta $ is an approximate identity, it follows that $ \hat{f_\delta}(\xi) \rightarrow 1 $ as $ \delta \rightarrow 0 $ for all $ \xi \in \R.$ Let $ h_\delta $ be the $K$-invariant function on $G/K$ constructed as in the theorem above using $ f_\delta $  in place of $ f.$ Then $ h_\delta $ will be supported in a small neighbourhood of the identity $ \tilde{h_\delta}(\lambda) \rightarrow 1$ as $ \delta \rightarrow 0.$
\end{rem}

Having  taken care of the construction for the symmetric space, we now consider the case $ \Omega = \C^n $ and $ P = L $ and prove the following result.
 
 \begin{thm}\label{splhermite-ingham} Assume that $ \theta$ is a positive decreasing function on $[0,\infty)$, vanishing at infinity for which $ \int_1^\infty \theta(t) \, t^{-1} dt < \infty.$ 
 Then for any $ \delta > 0 $ we can construct  a radial $ g_\delta $ supported on $ B(0,\delta) $ having the following two properties: (i) $ f \times g_\delta \rightarrow f $ in $ L^1(\C^n) $ in  $ L^1(\C^n) $ for any $ f \in L^1(\C^n) $ and (ii)  $ \| g_\delta \times  \varphi_k^{n-1}\|_2 \leq C e^{-\sqrt{2k+n} \,\theta_\delta(\sqrt{2k+n})} $ for any $ k \in \Na$ where $ \theta_\delta $ satisfies the same hypothesis as $ \theta.$
 \end{thm}
 
 In proving this theorem we make use of the following  result established in the context of the Heisenberg group. Let us briefly recall the setting before stating the result. We let $ \He^n = \C^n \times \R $ stand for the Heisenberg group whose group law is given by
 $$ (z,t)(w,s) = (z+w, t+s+\frac{1}{2} \Im( z \cdot \bar{w})) .$$
The convolution between two functions $ f $ and $ g $ on $ \He^n $ is defined by 
$$  f \ast g(z,t)  = \int_{\He^n} f(z-w, t-s- \frac{1}{2} \Im( z\cdot \bar{w})) g(w,s) dw ds. $$ 
We use the notation $ f^\lambda, \lambda \in \R $ to denote the inverse Fourier transform
$$  f^\lambda(z) = \int_{-\infty}^\infty f(z,t) e^{i \lambda t} dt $$
in the central variable. Then it is easy to see that 
$$ (f \ast g)^\lambda(z) = \int_{\C^n} f^\lambda(z-w) g^\lambda(w) e^{\frac{i}{2} \lambda \Im(z \cdot \bar{w})} dw .$$
The right hand side of the above is denoted by $ f^\lambda \ast_\lambda g^\lambda $ is called the $ \lambda$-twisted convolution of $ f^\lambda $ and $ g^\lambda.$ When $ \lambda = 1 $ we simply denote the $1$-twisted convolution between $ F $ and $ G $ on $ \C^n $ simply by $ F \times G.$

The Fourier transform of a function $ f \in L^1(\He^n) $ is the operator valued function defined on the set of all nonzero reals, $ \R^\ast $ given by
$$  \hat{f}(\lambda) = \int_{\He^n} e^{i\lambda t} f(z,t)  \pi_\lambda(z) dz dt = \int_{\C^n} f^\lambda(z) \pi_\lambda(z) dz $$ where $ \pi_\lambda(z) $ are the unitary operators acting on $ L^2(\R^n) $ by
$$ \pi_\lambda(x+iy)\varphi(\xi) = e^{i\lambda( x\cdot \xi+\frac{1}{2} x\cdot y)} \varphi(\xi+y).$$
For a function $ F $ on $ \C^n $ and $ \lambda \in \R^\ast $ we define the Weyl transform $ W_\lambda(F) $ by 
$$ W_\lambda(F) = \int_{\C^n} f^\lambda(z) \pi_\lambda(z) dz $$
so that $ \hat{f}(\lambda) = W_\lambda(f^\lambda).$ When $ \lambda =1 $ we write $ W(F) $ in place of $ W_1(F).$ In \cite[Theorem 4.5]{BGST} the authors have proved the following result.

\begin{thm}\label{heis-ingham}Assume that $ \theta$ is a positive decreasing function on $[0,\infty)$, vanishing at infinity for which $ \int_1^\infty \theta(t) \, t^{-1} dt < \infty.$ Then we can construct a compactly supported nontrivial radial function  $ f $ on $ \He^n$ whose Fourier transform satisfies the estimate
$$  \hat{f}(\lambda)^\ast \hat{f}(\lambda) \leq C e^{-2 \sqrt{H(\lambda)} \, \theta(\sqrt{H(\lambda)})} $$
 where $ H(\lambda) = -\Delta+\lambda^2 |x|^2 $ is the scaled Hermite operator.

\end{thm}

In \cite[Section 4]{BGST} the authors have constructed such a function with the added properties that $ f \geq 0 $ and $ \|f\|_1=1.$ It then follows by standard arguments that the family of dilated functions
$ f_\delta(z,t) = \delta^{-(2n+2)} f(\delta^{-1}z, \delta^{-2}t),\, \delta>0 $ is an approximate identity for $ L^1(\He^n)$: i.e. $  g \ast f_\delta \rightarrow g $  as $ \delta \rightarrow 0 $ in $ L^1(\He^n).$ \\

{\bf{Proof of Theorem \ref{splhermite-ingham}:}} We let $ \He^n_{red} = \He^n /\Gamma $ where $ \Gamma $ is the subgroup $ \{0\} \times 2\pi \Z $ of $ \He^n.$ Note that as a manifold, $\He^n /\Gamma $ can be identified with $ \C^n \times S^1 $ and functions on $\He^n /\Gamma $ are in one to one correspondence with functions on $ \He^n $ that are $ 2\pi$ periodic in the $t$-variable. Given $ f_\delta $ as above, we consider the periodisation
$$ F_\delta(z,t) = \sum_{k=-\infty}^\infty f_\delta(z,t+2k\pi) $$
which is a function on $ \He^n /\Gamma .$ Note that $ F_\delta \geq 0, \int_{\He^n /\Gamma } F_\delta(z,t) dz dt = 1 $ and
$$ \int_{\He^n /\Gamma } G((z,t)(w,s)^{-1}) F_\delta(w,s) dw ds = \int_{\He^n} G((z,t)(w,s)^{-1}) f_\delta(w,s) dw ds $$
for any $ G \in L^1(\He^n /\Gamma).$ The integral converges due to the fact that $ f_\delta $ is compactly supported.
From the above expression, it is easy to see that $ G \ast F_\delta \rightarrow G $ in $ L^1(\He^n /\Gamma).$ Since 
$$ \int_0^{2\pi} G\ast F_\delta(z,t) e^{it} dt = \int_{\C^n} G^1(z-w) f_\delta^1(w) e^{\frac{i}{2} \Im(z \cdot \bar{w})} dw $$
it follows that $ \|G^1 \times f_\delta^1 - G^1\|_1 \rightarrow 0 ,$  in view of the inequality
$$  \int_{\C^n} | G^1 \times f_\delta^1(z) - G^1(z)| dz \leq  \int_{\He^n/\Gamma} | G\ast f_\delta(z,t) - G(z,t)| dz dt .$$ 
The functions $ g_\delta(z) = f_\delta^1(z) $ give us the required family. It is clear that $ g \times g_\delta $ converges to $ g $ in $ L^1(\C^n) $ which can be seen by taking $ G(z,t) = g(z)e^{-it} $ in the above.

It remains to verify condition (ii) of Theorem \ref{splhermite-ingham}. As in the proof of  \cite[Theorem 4.6]{BGST} we can assume that $ \delta =1.$ Then the estimate  on $ \hat{f}(\lambda) $ gives us
 $$ W(g_1)^\ast W(g_1) = \hat{f}(1)^\ast \hat{f}(1) \leq C e^{-2 \sqrt{H} \,\theta(\sqrt{H})} .$$
 As $ f $ is radial, so is $ g_1 $ and it well known that 
 $$ W(g_1) = \sum_{k=0}^\infty  R_k(g_1) P_k, \,\, \,\,   W(g_1)^\ast W(g_1) = \sum_{k=0}^\infty  |R_k(g_1)|^2 P_k,$$
where $ R_k(g_1) $ are given by the equation $ g_1 \times \varphi_k^{n-1}(z) = R_k(g_1) \varphi_k(z).$ Consequently, the condition on $\hat{f}(1)^\ast \hat{f}(1)$ gives us
$$  \| g_1 \times \varphi_k^{n-1} \|_2^2 = c_n \frac{(k+n-1)!}{k!(n-1)!} |R_k(g_1)|^2 \leq C e^{-2 \sqrt{2k+n} \,\theta(\sqrt{2k+n})} .$$
This completes the proof of Theorem \ref{splhermite-ingham}.\\

{\bf{Proof of Theorem 1.3:}} The necessity of the condition $ \int_1^\infty \theta(t) t^{-1} dt <\infty $ for the existence of compactly supported $ f $ satisfying 
 \begin{equation} 
	 \|P_kf\|_2 \leq C e^{-\sqrt{\lambda_k}\theta(\sqrt{\lambda_k})}
\end{equation}
has been already proved. In fact, this part follows from  Chernoff's theorem  valid for any of the pairs $ ( \Omega, P) .$  The sufficiency part has been already verified for $ (X, \Delta) $ and $ (\C^n, L) $ in  the course of proofs  of Theorem \ref{ing-sym} and Theorem \ref{ing-sp-her}. Thus, we only need to construct compactly supported $ f $ with the prescribed decay in the Hermite case.\\

 We claim that to construct such a function,  we can make use of the function constructed for the Special Hermite case.  To substantiate our claim, we require  the following result about the Hermite projections of radial functions on  $\R^n$, proved in \cite[Theorem 3.4.1]{TH1}.  \\ 
 \begin{thm} \label{proj}
Suppose $f(x)=f_0(|x|)$ is radial. Then for any $ k $ 
$$P_{2k+1}(f)=0,\,\,\,  \ \ P_{2k}(f)(x)=R_{k}^{\frac{n}{2}-1}(f_0)L_k^{\frac{n}{2}-1}(|x|^2)e^{-\frac{|x|^2}{2}}$$ where  $R_{k}^{\frac{n}{2}-1}(f_0)$ are the Laguerre coefficients of $ f_0 $  given by 
$$ R_{k}^{\frac{n}{2}-1}(f_0)=2\frac{\Gamma(k+1)}{\Gamma(k+n/2 )}\int_{0}^{\infty}f_0(r)L^{\frac{n}{2}-1}_k(r^2)e^{-\frac{r^2}{2}}r^{n-1}dr.$$
 \end{thm}
 
 The construction of the required function is carried out in the following  proposition.
  
\begin{prop}\label{ hermite-ingham} Assume that $ \theta$ is a positive decreasing function on $[0,\infty)$, vanishing at infinity for which   $ \int_1^\infty \theta(t) \, t^{-1} dt < \infty.$ Then there exists a compactly supported continuous function $f$ on $\R^n$ satisfying  $  \|P_k f\|_2 \leq C e^{-  \theta(\sqrt{2k+n})\sqrt{2k+n}} $ for all $k\in\mathbb{N}. $ 
\end{prop}
\begin{proof}
	First we deal with the even dimensional case, $n=2m.$  Given $ \theta $ as in the proposition,  consider the function  $\Theta(t):= \sqrt{2}\, \theta(\sqrt{2} \,t), ~t>0. $ Then from the hypothesis, it is clear that $\int_{1}^{\infty}\Theta(t)t^{-1}dt<\infty.$ By Theorem \ref{splhermite-ingham} there exists a compactly supported radial function $g$ on $\C^m$ satisfying   the decay condition
	\begin{equation}
	\label{est_H}
	\frac{(k+m-1)!}{k!(m-1)!} |R_k(g)|^2  \leq C e^{-2 \sqrt{2k+m}\,\,\Theta(\sqrt{2k+m})}
	\end{equation} 
where we have used  the same notation  as in Theorem \ref{splhermite-ingham}, viz.
$$R_k(g)=\frac{ k! (m-1)!}{(k+m-1)!} \int_{\C^m}g(z)L_k^{m-1}\left(\frac{1}{2}|z|^2\right)e^{-\frac{1}{4}|z|^2}dz.$$  
Note that by a simple change of variable the above transforms into 
	$$R_k(g)=2^m\frac{   k! (m-1)!}{(k+m-1)!} \int_{\C^m}g(\sqrt{2}z)L_k^{m-1}\left( |z|^2\right)e^{-\frac{1}{2}|z|^2}dz.$$  
By defining  $f(x,y)=g(\sqrt{2} z),~ z= x+iy \in \C^m $ we see that  $f$ is radial and hence by Theorem \ref{proj} we have $P_{2k+1}f=0$ and 
$$P_{2k}(f)(x,y)= R_{k}^{\frac{n}{2}-1}(f_0)L_k^{\frac{n}{2}-1}(|z|^2)e^{-\frac{|z|^2}{2}} = c_n R_k(g) \varphi_k^{n-1}(\sqrt{2}z).$$   
Since $ \| \varphi_k^{n-1}\|_2^2 $ is a constant multiple of $ \frac{ k! (m-1)!}{(k+m-1)!} $ (see e.g. \cite[Proposition 3.4.1]{TH1}), we get
$$ \|P_{2k}f \|_2^2 = c_n^\prime \frac{(k+m-1)!}{k!(m-1)!} |R_k(g)|^2  \leq C e^{-2 \sqrt{2k+m}\,\,\Theta(\sqrt{2k+m})}   $$
where we have used the estimate \ref{est_H}. Finally recalling the definition of $\Theta $ and using 
$$\|P_{2k}f\|_2^2 \leq Ce^{-2\sqrt{4k +n} \, \theta(\sqrt{4k+n})} $$ proving the required Ingham type decay. This takes care of the even dimensional case.

When $ n = 2m+1 $ we construct a compactly supported function $  F(x,t) $ on $ \R^{n+1} $ satisfying the estimate 
$$   \| P_k F\|_2^2 \leq C e^{-2\sqrt{2k+n+1}\,\Theta(\sqrt{2k+n+1})} $$
where $ \Theta $ is defined in terms of $ \theta $ by the relation $ \theta(t) = \Theta(\sqrt{1+t^2}).$
If $ h_j(t) $ stand for the normalised Hermite functions on $ \R$ then the Hermite functions on $ \R^{n+1} $ are given by $ \Phi_{(\alpha,j)}(x,t) = \Phi_\alpha(x)h_j(t), \alpha \in \Na^n $ so that
$$ \|P_k F\|_2^2 = \sum_{j=0}^k \sum_{|\alpha| = k-j} |( F, \Phi_{(\alpha,j)}) |^2.$$
By defining $ f $ on $ \R^n $ by $ f(x) = \int_{-\infty}^\infty F(x,t) h_0(t) dt $ we  get a compactly supported function for which 
$$  \|P_k f\|_2^2 =   \sum_{|\alpha| = k} |( F, \Phi_{(\alpha,0)}) |^2 \leq   \| P_k F\|_2^2 \leq C e^{- 2\sqrt{2k+n+1}\,\Theta(\sqrt{2k+n+1})} .$$ 
In view of the relation between $ \Theta $ and $ \theta,$ the function $ P_k f $ has  the required decay.
\end{proof}

\section*{Acknowledgments} The authors are immensely thankful to the referee for his meticulous reading of the manuscript  and making very useful suggestions  which have greatly improved the exposition. The first author is supported by Int. Ph.D. scholarship from Indian Institute of Science.
The second author is supported by  J. C. Bose Fellowship from the Department of Science and Technology, Government of India.
%%%%%%%%%%%%%%%%%%%%%%%%%%%%%%%%%%%%%%%%%%%%%%%%%%%%%%

\end{document}